\documentclass[12pt]{amsart}
\usepackage[utf8]{inputenc}
\usepackage[T1]{fontenc}
\usepackage{amsthm,amsmath,amssymb}
\usepackage{cleveref}
\usepackage{xcolor}
\usepackage[all]{xy}
\input{grcawol.sty}
\usepackage{listings}
 \newtheorem{Thm}{Theorem}[section]
 \newtheorem{Lem}[Thm]{Lemma}
 \newtheorem{Prop}[Thm]{Proposition}
 \newtheorem{Cor}[Thm]{Corollary}
 
\theoremstyle{remark}

 \newtheorem{Rem}[Thm]{Remark}

\theoremstyle{definition}
 \newtheorem{Def}[Thm]{Definition}
\numberwithin{equation}{section}

\newcommand\ol[1]{\overline{#1}}

\newcommand\braid{\sigma}
\newcommand\tr{\operatorname{tr}}
\newcommand\End{\operatorname{End}}

\newcommand\hit{\triangleright}

\newcommand\bhit{\blacktriangleright}
\newcommand\inv{^{-1}}
\def\HM#1.#2.#3.#4.{{^{#1}_{#3}\mathcal M^{#2}_{#4}}}

\newcommand\id{\operatorname{id}}

\newcommand\ot{\otimes}

\newcommand\Vect{\operatorname{Vec}}

\newcommand\ptr{\operatorname{ptr}}

\newcommand\Q{\mathbb Q}
\newcommand\CC{\mathbb C}
\newcommand\CCu{\CC^\times}
\newcommand\ZZ{\mathbb Z}

\newcommand\C{\mathcal C}
\newcommand\CTR{\mathcal Z}

\newcommand\semdir{\rtimes}

\newcommand\semidir\rtimes

\newcommand\Gal{\operatorname{Gal}}
\newcommand\ab{{\operatorname{ab}}}
\newcommand{\gb}{\overline{g}}
\newcommand{\hb}{\overline{h}}
\newcommand{\kb}{\overline{k}}
\begin{document}
\title[Invariants of Modular Tensor Categories]{A topological invariant for modular fusion categories}
\author{Ajinkya Kulkarni}
\email{ajinkya.kulkarni@u-bourgogne.fr}
\author{Michaël Mignard}
\email{mignard@math.cnrs.fr}
\author{Peter Schauenburg}
\email{peter.schauenburg@u-bourgogne.fr}
\address{Institut de Math{\'e}matiques de Bourgogne, UMR 5584 CNRS
\\
Universit{\'e} Bourgogne Franche-Comté\\
F-21000 Dijon\\France}


\begin{abstract}
The modular data of a modular category $\C$, consisting of the $S$-matrix and the $T$-matrix, is known to be an incomplete invariant of $\C$. More generally, the invariants of framed links and knots defined by a modular category as part of a topological quantum field theory can be viewed as numerical invariants of the category. Among these invariants, we study the invariant defined by the Borromean link colored by three objects. Thus we obtain a tensor that we call $B$. We derive a formula for the Borromean tensor for the twisted Drinfeld doubles of finite groups. Along with $T$, it distinguishes the $p$ non-equivalent modular categories of the form $\CTR(\Vect_G^\omega)$ for $G$ the non-abelian group $\ZZ/q\ZZ \semdir \ZZ/p\ZZ$, which are not distinguished by the modular data.
\end{abstract}
\maketitle

\section{Introduction}
The modular data of a modular category $\C$ comprises the $S$- and $T$-matrices, two square matrices indexed by the isomorphism classes of simples; they define a projective representation of the modular group. On the one hand one could say that these two matrices are just particular instances of the topological invariants defined by a modular category in the framework of a topological quantum field theory: the $S$-matrix is the invariant defined for a Hopf link colored by two simples of the category, and the $T$-matrix contains the components of a kink. On the other hand, one may feel that the modular data is somewhat privileged among the topological invariants associated to the category: Invertibility of the $S$-matrix already features in the definition; invariance properties with respect to the modular group are key for the appearance of modular categories in conformal field theory; last but not least properties of the modular data are important for the purely algebraic study of modular categories. The importance of the modular data has led to the question being seriously considered (and stated as not quite a conjecture in \cite{MR3486174}) as to whether a modular category (and hence the TQFT associated to it) is already determined fully by the modular data. This was refuted in \cite{2017arXiv170802796M} by a family of examples that are taken among the particularly accessible class of group-theoretical modular categories, more specifically the Drinfeld centers of pointed fusion categories, which were already considered, in the guise of representation categories of twisted Drinfeld doubles of finite groups, in \cite{MR1128130}. It turns out in fact that arbitrarily many inequivalent modular categories can give rise to the same modular data; the examples are defined by the same noncommutative group, endowed with different three-cocycles; the smallest example in \cite{2017arXiv170802796M} concerns the nonabelian group of order $55$, although it is not known whether smaller examples exist.

This result naturally gives rise to the following question: Can one find other invariants that distinguish the modular categories in these families? Note that this was not how the categories were found to be distinct in \cite{2017arXiv170802796M}, where rather the inexistence of suitable equivalences was proved through the characterization of such equivalences via Morita equivalence of pointed fusion categories. Recent general results on the correspondence between modular categories and topological quantum field theories \cite{2015arXiv150906811B} imply that the entire extended TQFT defined by the category can be viewed as a complete invariant. This does not, of course, solve the concrete problem of finding invariants that one can compute for specific categories and use to distinguish them ---the entire TQFT is a rather formidable collection of data.

The simple idea of the present paper is to consider the invariant of a certain framed link defined by the modular categories in question and view it as an invariant of the category, much like the invariant of the Hopf link giving the $S$-matrix. The particular link we will use is known as the borromean rings. This is partly an obvious candidate for naive reasons: It is the closure of a three-strand braid, which is the next more complicated thing over the two-strand braid whose closure is the Hopf link; having three strands might allow the associativity constraint of the category (which, after all, is encoded in the three-cocycle that makes the basic difference in the aforementioned examples)  to have a greater influence on the invariant. Since the invariant obtained from a full twist on three strands (like the Hopf link comes from a full twist on two strands) is easily seen to be determined by the modular data, making inverse braidings appear seems necessary, and the braid whose closure gives the borromean rings does this in a somewhat symmetric fashion. It may also be a good candidate for a slightly less naive reason: The borromean rings are three circles that are pairwise not linked, yet form a nontrivial link. This is a somewhat subtle topological phenomenon which one may hope gives rise to an invariant whose properties are \emph{not} covered by those of the $S$-matrix, which records precisely what happens if two rings are linked. (In fact, the rings seem to have appeared in 15th century Italy as a heraldic symbol for this very reason: They are supposed to symbolize the political (and marital) alliances between the Borromeo, Sforza, and Visconti families, which was such that removing any one of them would have broken the alliance of the three.) Whether the motivations are justified by the success is perhaps doubtful: The borromean tensor does not at all distinguish the categories described above. In fact it does not seem to ``see'' the three-cocycle at all that makes the difference between the categories. It is only the $T$-matrix taken together with the $B$-tensor that makes it impossible to find a bijection between the simples of the different categories that would map these data to each other.

The same general idea of using link invariants to distinguish modular categories not distinguished by modular data was also pursued in \cite{2018arXiv180505736B}, where the authors show that the invariant of the Whitehead link, along with the $T$-matrix, does distinguish the five inequivalent modular categories defined from the nonabelian group of order $55$. We are grateful to the authors for letting us see an advance copy of their preprint. At the time we knew by computer experiments that the invariant of the borromean rings, taken together with the modular data, also distinguishes the categories in this particular example, and we knew which components of the ``borromean tensor'' (given by the invariants of the borromean link with its three components colored by three simples of the category) are responsible for this success. We had not finished writing our findings, and we had not completed the results in \cref{sec:Example} showing that the $T$-matrix together with the borromean tensor is sufficient to distinguish the modular categories associated to the nonabelian groups of order $pq$ (for all primes $p,q$ for which such a group exists).

The paper is organized as follows: After introducing conventions and notations, we first revisit the modular data of twisted Drinfeld doubles to give a slightly improved formula for the $S$-matrix, but mostly to introduce the methods to be used later. We formally define the Borromean tensor in \cref{sec:borromean-tensor} and record some symmetry properties it enjoys. In \cref{sec:borr-tens-twist} we give an explicit formula for the Borromean tensor for twisted Drinfeld doubles, with some useful specializations that we then use in \cref{sec:Example} to explicitly distinguish the inequivalent modular categories found in \cite{2017arXiv170802796M} by the new numerical invariant that is the $T$-matrix together with the $B$-tensor. An appendix gives some GAP codes for computing the $B$-tensor (and the $S$-matrix). Experimenting with computer calculation was an important step in our work, although computer help is not needed to prove the main result; some calculations were performed using HPC resources from PSIUN CCUB (Centre de Calcul de
l'Université de Bourgogne).

\section{Preliminaries}

Throughout the paper we will consider modular categories, that is, braided spherical $\CC$-linear fusion categories $\C$ such that the square matrix $S$ whose coefficients are the traces of the square of the braiding on pairs of simple objects is invertible. We refer to \cite{BakKir:LTCMF,EtiNikOst:FC,MR3242743} for background. We will denote by $(X_i)_{i\in I}$ a set of representatives of the isomorphism classes of simple objects of $\C$, and write $i^*\in I$ for the element such that $X_{i^*}=X_i^*$ is the dual object. We will denote the pivotal trace in the category $\C$ of an endomoprhism $f$ by $\ptr(f)$.

One \emph{raison d'être} of modular categories is that they allow the definition of a topological quantum field theory. In particular they define invariants in $\CC$ of framed knots and links. We will freely use graphical notation for morphisms in a modular category $\C$. The framed link invariant defined by a modular category can be viewed as follows: The link is the closure of a braid. The braid in question, colored by objects in $\C$,  defines an endomorphism of a tensor product of objects in $\C$, and taking the closure of the braid corresponds to taking the (pivotal) trace of the endomorphism. To fix notations regarding this procedure, there is a representation of the braid group on $n$ strands
$$R\colon\mathbb B_n\to\operatorname{Aut}_{\C}\bigl( (\dots(V\ot V)\ot V)\ot V)\dots \ot V)\bigr)$$
on the tensor product of $n$ copies of any object $V\in\C$. In the case of a non-strict category (as indicated by the parentheses) this involves both instances of the braiding $\braid$, and instances of the associator isomorphism $\Phi$. We will also need to consider analogous morphisms defined on tensor products of distinct objects; informally we will write
\begin{multline*}
  R(\beta)\colon (\dots((V_1\ot V_2)\ot V_3)\dots\ot V_n)\\\rightarrow (\dots((V_{\beta\inv(1)}\ot V_{\beta\inv(2)})\ot V_{\beta\inv(3)})\dots\ot V_{\beta\inv(n)}) 
\end{multline*}
for $\beta\in\mathbb B_n$ and $V_1,\dots,V_n\in\C$, where $\beta$ also denotes the underlying permutation of $\beta$. We note that  $R(\beta\beta')=R(\beta)R(\beta')$ in the obvious sense for two braids $\beta,\beta'$.

The modular categories in which we will do extensive calculations are the Drinfeld centers of pointed fusion categories; an alternative description of these is as module categories of twisted Drinfeld doubles of finite groups \cite{MR1128130}, although these quasi-Hopf algebras will play no explicit role in this paper. We need to fix notations and collect a few useful identities.

We will write $g\hit h=ghg\inv$ for the action of a group on itself by conjugation, $C_G(g)$ for the centralizer of $g$ in $G$, and $g^G$ or $\ol g$ for the conjugacy class of $g\in G$. Any pointed fusion category is of the form $\Vect_G^\omega$, the category of finite-dimensional $G$-graded vector spaces with associativity constraint given by
\begin{align*}
  \Phi\colon (U\ot V)\ot W&\longrightarrow U\ot (V\ot W)\\
             (u\ot v)\ot w&\longmapsto \omega(|u|,|v|,|w|)u\ot(v\ot w) 
\end{align*}
for $U,V,W\in \Vect_G^\omega$ and homogeneous elements $u\in U,v\in V,w\in W$, where $|u|$ is our notation for the degree of the homogeneous element $u$, and $\omega\colon G^3\rightarrow\CCu$ is a three-cocycle. In the sequel, we will always tacitly assume that elements of graded vector spaces are homogeneous in writing such formulas.

The category $\Vect_G^\omega$, and by extension its Drinfeld center below, is spherical with respect to the canonical pivotal structure characterized by the property that pivotal dimensions coincide with the usual vector space dimensions. We will not explicitly need the pivotal structure, but only the fact that pivotal traces of endomorphisms are simply the usual traces of the underlying linear maps.

The (right) Drinfeld center $\mathcal{Z}(Vec^\omega_G)$ is a modular category. The structure of an object $(W,\sigma_{\cdot,W})$ in the center, with $\sigma_{V,W}\colon V\ot W\to W\ot V$ the half-braiding for $V\in\Vect_G^\omega$ can be described in terms of an action (not quite a group action) of $G$ on $W$ \cite{Maj:QDQHA}. More precisely, giving a half-braiding is equivalent to giving a map
$\hit\colon G\otimes W\to W$ subject to the conditions
\begin{align}
  |g\hit w|&=g\hit|w|\\
  e\hit w&=w\\
  g\hit h\hit w&=\alpha_{|w|}(g, h)gh\hit w
\end{align}
for $g,h\in G$ and $w\in W$, the equivalence being described by the formula
\begin{equation*}
  V\ot W\ni v\ot w\mapsto |v|\hit w\ot v\in W\ot V.
\end{equation*}

Let $V\in\CTR(\Vect_G^\omega)$. Since acting by an element $g\in G$ is a vector space automorphism of $V\in\Vect_G^\omega$ which conjugates degrees, any object decomposes as the direct sum of objects where the degrees of nonzero homogeneous components form a conjugacy class. Assume that the degrees of the nonzero components of $V$ form a conjugacy class. Then $G$ permutes those homogeneous components transitively, and elements in the centralizer $C_G(g)$ map the homogeneous component $V_g$ to itself. In particular, $V_g$ is an $\alpha_g$-projective representation of $C_G(g)$ where
\begin{equation}\label{alpha}
\alpha_g(x,y)=\omega(x,y,g)\omega^{-1}(x,y \hit g,y)\omega(xy \hit g ,x ,y)
\end{equation}
and the structure of $V$ is determined by this projective representation for any one $g$ in the conjugacy class. In particular, simple objects of $\CTR(\Vect_G^\omega)$ are parametrized by pairs $(g,\chi)$ where $g$ runs through a system of representatives for the conjugacy classes of $G$, and $\chi$ is an irreducible $\alpha_g$-projective character of $C_G(g)$.

It is convenient to note how to obtain the $C_G(x)$-projective character $\chi'$ describing the action of $C_G(x)$ on $V_x$ when $x$ belongs to the same conjugacy class, say $x=f\hit g$. So let $c\in C_G(x)$ and $v\in V_g$. We have

\begin{align}\label{1}
c \hit (f \hit v) &= \alpha_{g}(c, f) cf \hit v\\
f \hit ((f\inv\hit c) \hit v)   &= \alpha_{g}(f, f^{-1}\hit c) cf \hit v
\end{align}
and therefore
\begin{equation*}
  c\hit f\hit v  = \alpha_{g}(c, f) \alpha_{g}\inv(f, f^{-1}\hit c) f \hit (f\inv\hit c)\hit v .
\end{equation*}
This means that the diagram
\begin{equation*}
\xymatrix {V_g \ar[rr]^{f\hit} \ar[dd]_{(f^{-1} \hit c)\hit}   & & V_x \ar[dd]^{c\hit}\\
\\
 V_g \ar[rr]^{f\hit}  & & V_x}
\end{equation*}
commutes up to the scalar factor $\alpha_{g}(c, f) \alpha_{g}\inv(f, f^{-1}\hit c)$. By cyclicity of the trace, the projective character of the projective $C_G(x)$-representation $V_x$ is therefore $\chi^x$, given by
\begin{equation}\label{chi}
  \chi^x:=(f\hit\chi)(c):=\alpha_{g}(c, f) \alpha_{g}\inv(f, f^{-1}\hit c)\chi(f\inv\hit c).
\end{equation}
In particular this expression does define a projective character, and does not depend on the choice of $f\in G$ with $f\hit g=x$.

The inverse of the braiding in $\CTR(\Vect_G^\omega)$ is given by $\sigma\inv(w\ot v)=v\ot |v|\inv\bhit w$, where $\bhit\colon\CC G\otimes V\to V$ is such that $g\hit g\inv\bhit v=g\inv\bhit g\hit v=v$. From
\begin{align*}
  f\inv \hit (f \hit v) &= \alpha_{|v|}(f\inv, f) v,\\
  f\hit f\inv\hit v&=\alpha_{|v|}(f,f\inv)v
\end{align*}
one reads off
\begin{align}\label{blackhit}
f\inv\bhit v&=\alpha_{|v|}\inv(f,f\inv) f\inv\hit v,\\
  f\inv\bhit v&=\alpha_{f\hit|v|}\inv(f\inv,f) f\inv\hit v,
\end{align}
respectively.

\section{The modular data of $\CTR(\Vect_G^\omega)$}
It is of course well-known how to compute the modular data of the twisted Drinfeld double of a finite group \cite{MR1770077}. In particular the $T$-matrix is given by
\begin{equation}
  \label{eq:tmatrix}
  T_{(g,\chi)}=\Theta(g,\chi)=\frac{\chi(g)}{\chi(1)}.
\end{equation}
We will rederive a formula for the $S$-matrix with only a slight advantage: The formula from \cite{MR1770077} involves a double sum, over two conjugacy classes, or twice over the group. Our formula has only one sum. The $S$-matrix is the trace of a braid on two strands (whence the two sums, related to the two objects coloring the strands); generally, the invariant obtained from taking the trace of a braid on $n$ strands would involve $n$ summations, over the conjugacy classes associated to the objects, but one can get away with only $n-1$ summations by a simple trick based on a well-known fact.

In the graphical calculus, taking the trace of the image of a braid in a pivotal monoidal category amounts to closing the braid. Obviously, one can choose to close all strands but one, which leaves us with an endomorphism with the object labelling the remaining strand, and then take the trace of that endomorphism. More formally:
\begin{Rem}\label{pretrick}
  Let $V,W\in\C$ and $f\colon V\ot W\to V\ot W$. Then
  $\ptr(f)=\ptr(\ptr_V(f))$, where
  \begin{equation}
    \ptr_V(f)=
    \gbeg36
    \gvac2\got1W\gnl
    \gdb\gcl1\gnl
    \gcl1\glmptb\gnot f\grmptb\gnl
    \gev\gcl1\gnl
    \gvac2\gob1W
    \gend
  \end{equation}
  is a partial (pivotal) trace. If $W$ is simple, then $\ptr_V(f)=\lambda\cdot\id_W$ is a scalar, and $\ptr(f)=\lambda\dim(W)$.
\end{Rem}

Working in the category $\CTR(\Vect_G^\omega)$, where formulas for the traces of braids involve sums over all the combinations of $G$-degrees of each object (subject to some condition), this will allow us to get away with one less sum (or, when coding the formulas, one less nested loop):
\begin{Rem}\label{trick}
  Let $V,W\in\CTR(\Vect_G^\omega)$ with $W$ the simple object corresponding to $(g,\chi)$, and let $f\colon V\ot W\to V\ot W$. Then $\tr_V(f)=\lambda\id$; the scalar $\lambda$ is determined by the component $\tr_V(f)|_{W_g}\colon W_g\to W_g$ of $\tr_V(f)$ by $\lambda\dim W_g=\tr(\tr_V(f)|_{W_g})$, and thus $\tr(f)=|\overline g|\tr(\tr_V(f)|_{W_g})$
\end{Rem}

As an illustration and warm-up for the calculations in \cref{sec:borr-tens-twist} we will consider the $S$-matrix for two simple objects $V,W\in\mathcal{Z}(\Vect^\omega_G)$, corresponding to the pairs $(g,\chi_1)$ and $(h, \chi_2)$. We need to compute
\begin{equation*}
  S_{(g,\chi_1),(h,\chi_2)}=\tr(\sigma_{WV}\sigma_{VW})=\tr(\sigma^2)=\tr(R(\sigma^2)),
\end{equation*}
where (no parentheses being necessary on two objects) there is no difference between the representation $R$ of the braid group $\mathbb B_2$ on one generator $\sigma$ and simply instances of the braiding of the category $\CTR(\Vect_G^\omega)$.

If we write $V = \underset{x\in \overline{g}}\oplus V_x$ and $W = \underset{x\in \overline{h}}\oplus W_y$, then for $  v\in V_x$ and $  w\in W_y$ we have
\begin{align*}
  \braid^2(  v\ot  w)&=\braid(x\hit   w\ot   v)\\
  &=|x {\hit}  w|\hit   v\ot x\hit   w\\
  &=(x\hit y) \hit   v\ot x\hit  w.
\end{align*}
We can endow $V\ot W$ with a $G\times G$-grading composed of the $G$-gradings of $  V$ and $  W$. Then
\begin{align*}
  \deg_{G\times G}\braid^2(v\ot w)=((x\hit y) \hit x,x \hit y)
\end{align*}
for $x=|v|,y=|w|$.

For a finite group $\Gamma$, a $\Gamma$-graded vector space $E$, and an endomorphism $f$ of $E$ let $f_0$ be trivial component of $f$ with respect to the $\Gamma$-grading of $\End(E)$. Then $\tr(f)=\tr(f_0)$.

In our example, considering the $G\times G$-grading of $V\ot W$, we see that $\braid^2( v\ot  w)$ has the same degree as $v\ot w$ if and only if $x$ and $y$ commute.

\begin{equation} \label{degree condition}
(\braid^2)_0(v\ot w) = \begin{cases} 
          x\hit v\ot y\hit  w & [x,y]=1 \\
          0 & [x,y] \neq 1 
       \end{cases}
\end{equation}

In particular
\begin{equation*}
  S_{V,W}=\tr(\braid^2)=\tr((\braid^2)_0)=\sum_{\substack{x\in \overline g\\y\in\overline h\\ [x,y]=e}}\chi_1^x(y)\chi_2^y(x)
\end{equation*}
Using \cref{trick} we can replace the double sum by a single sum; also, we can use~(\cref{chi}):
\begin{align*}
S_{(g, \chi_1),(h, \chi_2)} &= |\hb|\sum_{x \in \gb}^{[x,h]=1} \chi_1^x (h) \chi_2 (x)\notag\\
&=|\hb| \sum_{x \in \gb}^{[x,h]=1} \alpha_{g}(c, p) \alpha_{g}\inv(p, p\inv\hit c)\chi_1(p\inv\hit c) \chi_2 (x)
\end{align*} 
where $p$ stands for any group element satisfying $p\hit g=x$. Alternatively, we can use the bijection $G/C_G(g) \rightarrow g^G$ given by $aC_G(g)\mapsto a\hit g $ to rewrite
\begin{align*}\label{Sgroup}
  S_{(g, \chi_1),(h, \chi_2)} &=\frac{|\hb|}{|C_G(g)|}\sum^{[p\hit g,h]=1}_{p\in G} \alpha_{g}(h, p) \alpha^{-1}_{g}(p, p^{-1}\hit b)\chi_1(p^{-1}\hit b) \chi_2 (p\hit g)\\
  &={|\hb|}\sum^{[p\hit g,h]=1}_{p\in G/C_G(g)} \alpha_{g}(h, p) \alpha^{-1}_{g}(p, p^{-1}\hit b)\chi_1(p^{-1}\hit b) \chi_2 (p\hit g)
\end{align*}
As mentioned, this formula is (up to conventions) quite like the formula in \cite{MR1770077}, except for two details: We have a single sum over one conjugacy class instead of a double sum, and we have half the cocycle (``$\alpha$'') terms due to the fact that we need to use (\cref{chi}) on only one of the two objects.

\section{The Borromean tensor}\label{sec:borromean-tensor}

A modular category (in fact any spherical braided fusion category) defines a numerical invariant of framed knots and links which can be written as the pivotal trace of the image in the category of a braid whose closure is the link, with its components colored by simple objects of the category. Read differently, each fixed framed link defines a numerical invariant of modular categories in this fashion. More precisely, the invariant is then indexed by as many simple objects as the link has components.

Among this infinite supply of numerical invariants (among which the $S$-matrix and, up to a dimension factor, the $T$-matrix can also be found) we pick one example, for the heuristic (and art historical) reasons cited in the introduction:

\begin{Def}
  The borromean tensor (or $B$-tensor) of a modular category with simples $(X_i)_{i\in I}$ is the family
  \begin{equation}
    \label{eq:borromeo}
    B_{ijk}:=\ptr(B((\braid_2\inv\braid_1)^3)
  \end{equation}
  where $B((\braid_2\inv\braid_1)^3)\in\operatorname{Aut}_{\C}((X_i\ot X_j)\ot X_k)$. Graphically
  \begin{equation*}
    B_{ijk}=\ptr\left(
    \gbeg38
    \got1i\got1j\got1k\gnl
    \gbr\gcl1\gnl
    \gcl1\gibr\gnl
    \gbr\gcl1\gnl
    \gcl1\gibr\gnl
    \gbr\gcl1\gnl
    \gcl1\gibr\gnl
    \gob1i\gob1j\gob1k\gend
  \right)
  =
  \gbeg6{13}
  \gwdb6\gnl
  \gcl{10}\gwdb4\gcl2\gnl
  \gvac1\gcl8\gnot{\qquad\; i}\gdb\gnot{\;j}
  \gcl1\gnot{\;k}\gvac1\gnl
  \gvac2\gcl6    \gbr\gcl1\gnl
  \gvac3  \gcl1\gibr\gnl
  \gvac3  \gbr\gcl1\gnl
  \gvac3  \gcl1\gibr\gnl
  \gvac3  \gbr\gcl1\gnl
  \gvac3  \gcl1\gibr\gnl
  \gvac2\gev\gcl1\gcl2\gnl
    \gvac1\gwev4\gnl
    \gwev6
  \gend
\end{equation*}
(In the graphical representation, we let $i$ stand for $X_i$.)
\end{Def}
\begin{Lem}
 The following equalities hold:
\begin{enumerate}
\item
$B_{ijk}=B_{jki}=B_{kij}$
\item
$B_{ijk}=B_{jik^{\ast}}$
\item
$B_{ijk}=\overline{B_{kji}}$ if $\C$ statisfies the $\mathcal F$-property (see \cite{MR2832261}).
\end{enumerate}
\end{Lem}
\begin{proof}
Clearly cyclicity of the trace implies that
the $B$-tensor is invariant with respect to cyclic permutations of its three indices. We need to show the other symmetry properties:

Conjugating with
$\gbeg33 \gbr\gcl1\gnl\gcl1\gbr\gnl\gbr\gcl1\gend$ shows that
\begin{equation}\label{180}
  \tr\left(
    \gbeg38
    \got1i\got1j\got1k\gnl
    \gbr\gcl1\gnl
    \gcl1\gibr\gnl
    \gbr\gcl1\gnl
    \gcl1\gibr\gnl
    \gbr\gcl1\gnl
    \gcl1\gibr\gnl
    \gob1i\gob1j\gob1k\gend
  \right)
  =\tr\left(
    \gbeg38
    \got1k\got1j\got1i\gnl
    \gcl1\gbr\gnl
    \gibr\gcl1\gnl
    \gcl1\gbr\gnl
    \gibr\gcl1\gnl
    \gcl1\gbr\gnl
    \gibr\gcl1\gnl
    \gob1k\gob1j\gob1i\gend
    \right)
\end{equation}

Furthermore, we have

\begin{equation*}
  \gbeg47
  \got1j\got1i\gdb\gnl
  \gcl
  1\gibr\gcl5\gnl
  \gbr\gcl1\gnl
  \gnot{\!\!k\phantom{k}}\gcl1\gibr\gnl
  \gbr\gcl1\gnl
  \gcl1\gibr\gnl
  \gob1i\gob1j\gev
  \gend
  =
  \gbeg47
  \got1j\gvac2\got1i\gnl
  \gcl1\gdb\gcl1\gnl
  \gbr\gbr\gnl
  \gnot{k\phantom{k}}\gcl1\gibr\gcl1\gnl
  \gbr\gbr\gnl
  \gcl1\gev\gcl1\gnl
  \gob1i\gvac2\gob1j\gend
  =
  \gbeg47
  \gdb\got1j\got1i\gnl
  \gcl5\gibr\gcl1\gnl
  \gvac1\gcl1\gbr\gnl
  \gnot{k\phantom{k}}\gvac1\gibr\gcl1\gnl
  \gvac1\gcl1\gbr\gnl
  \gvac1\gibr\gcl1\gnl
  \gev\gob1i\gob1j\gend
\end{equation*}

and thus

\begin{equation*}
  \tr\left(
    \gbeg38
    \got1i\got1j\got1k\gnl
    \gbr\gcl1\gnl
    \gcl1\gibr\gnl
    \gbr\gcl1\gnl
    \gcl1\gibr\gnl
    \gbr\gcl1\gnl
    \gcl1\gibr\gnl
    \gob1i\gob1j\gob1k\gend
  \right)
  =\tr\left(
    \gbeg48
    \got1i\got1j\gdb\gnl
    \gbr\gcl1\gcl6\gnl
    \gcl1\gibr\gnl
    \gbr\gcl1\gnl
    \gnot{\!\!k\phantom k}\gcl1\gibr\gnl
    \gbr\gcl1\gnl
    \gcl1\gibr\gnl
    \gob1i\gob1j\gev\gend
  \right)
  =\tr\left(
    \gbeg48
    \gvac2\got1i\got1j\gnl
    \gdb\gbr\gnl
    \gcl5\gibr\gcl1\gnl
    \gvac1\gcl1\gbr\gnl
    \gnot{\!\!k\phantom k}\gvac1\gibr\gcl1\gnl
    \gvac1\gcl1\gbr\gnl
    \gvac1\gibr\gcl1\gnl
    \gev\got1i\got1j\gend
  \right)
  =\tr\left(
    \gbeg38
    \got1{k^*}\got1i\got1j\gnl
    \gcl1\gbr\gnl
    \gibr\gcl1\gnl
    \gcl1\gbr\gnl
    \gibr\gcl1\gnl
    \gcl1\gbr\gnl
    \gibr\gcl1\gnl
    \gob1{k^*}\gob1i\gob1j\gend
    \right)
\end{equation*}

Combining with \eqref{180} we see that $B_{ijk}=B_{jik^*}$. Together with the cyclic permutation invariance, this yields that $B$ is invariant under any permutation of its indices combined with dualizing an even or odd number of its indices according to the parity of the permutation. If the borromean braid has finite order in the category in question (as is the case for group-theoretical modular categories, see [reference]), then $B_{ijk}$ is also the complex conjugate of the trace of the inverse Borromean braid. But the right hand side of \eqref{180} is that trace, so $B_{ijk}=\overline{B_{kji}}$. So in this case $B$ is invariant under permutation and dualization of its indices, up to conjugation if the number of dualized indices has the opposite parity of the permutation.
\end{proof}

For larger rank categories, these symmetry properties could serve to speed up the computation of the borromean tensor, although, truth be told, we have so far only used them to debug our code.
\section{The borromean tensor of a twisted double}\label{sec:borr-tens-twist}
\newcommand\Pe{P}

In this section we will derive an explicit formula for the borromean tensor in the Drinfeld center of a pointed fusion category, in terms of the group, the cohomological data, and the projective characters parametrizing the simple objects. It should be noted that in principle it is known how to obtain such formulas for the topological invariants defined by braided monoidal categories. Nevertheless it is a rather tedious undertaking to provide them in complete detail.

Consider three simple objects $ U, V, W \in  \mathcal{Z}(Vec^\omega_G)$, parametrized by couples $(g, \chi_1)$, $(h, \chi_2)$ and $(k, \chi_3)$. Take $u\in U_x, v\in V_y,  w\in W_z$, where $x\in \gb, y\in \hb, z\in \kb$.

Looking at the $G^3$-degree of tensors in $(U\ot V)\ot W$ (which is not affected by associativity isomorphisms), we see that for $u\in U_x$, $v\in V_y$, $w\in W_z$ we have
$\deg_{G^3}(R(\braid_2\inv\braid_1))(u\ot v\ot w)=\Pe(x,y,z)$ if we define $\Pe\colon G^3\to G^3$ by $\Pe(x,y,z)=(x\hit y,z,z\inv\hit x)$. Note $\Pe\inv(x,y,z)=(y\hit z,y\hit z\inv\hit x,y)$. Now in order that $((\sigma_2\inv\sigma_1)^3)_0(u\ot v\ot w)\neq 0$, we need to have $\Pe^3(x,y,z)=(x,y,z)$, or equivalently $\Pe^2(x,y,z)=\Pe\inv(x,y,z)$. Comparing
\begin{align*}
  \Pe^2(x,y,z)&=\Pe(x\hit y,z,z\inv\hit x)\\
            &=((x\hit y)\hit z,z\inv\hit x,(z\inv\hit x\inv)\hit x\hit y)\\
            &=((x\hit y)\hit z,z\inv\hit x,[z\inv,x\inv]\hit y)\\
  \Pe\inv(x,y,z)&=(y\hit z,y\hit z\inv\hit x,y)
\end{align*}
we see that:
\begin{align}
  ((\braid_2\inv\braid_1)^3)_0&|_{U_x\ot V_y\ot W_z}\neq 0\notag\\
                            & \Leftrightarrow
                              \begin{cases}(x\hit y)\hit z=y\hit z\\
                                z\inv\hit x=y\hit z\inv\hit x\\
                                [z\inv,x\inv]\hit y=y
                              \end{cases}\notag\\
                            &\Leftrightarrow\label{eq:condition0}
                              \begin{cases}[[y\inv,x],z]=1\\
                                [[z,y],x]=1\\
                                [[z\inv,x\inv],y]=1
                              \end{cases}
\end{align}

We evaluate the morphism $R((\braid_2\inv\braid_1)^3)$ in three steps.
\begin{align*}
  R(\braid_2\inv\braid_1)&((u\ot v)\ot w)\\
                         :&=\Phi(V\ot\sigma\inv)\Phi\inv(\sigma\ot W)((u\ot v)\ot w)\\
  &=\Phi(V\ot\sigma\inv)\Phi\inv((x\hit v\ot u)\ot w)\\
  &=\Phi(V\ot\sigma\inv)(\omega(x\hit y,x,z)(x\hit v\ot(u\ot w))\\
  &=\Phi(\omega(x\hit y,x,z)(x\hit v\ot (w\ot z\inv\bhit u)\\
                         &=\omega\inv(x\hit y,z,z\inv\hit x)\omega(x\hit y,x,z)(x\hit v\ot w)\ot z\inv\bhit u\\
                         &=\psi(x,z\inv\hit x,x\hit y,z)\Psi((u\ot v)\ot w)
 \end{align*}
 with
 \begin{align*}
   \psi(x,x',y,z)&=\omega\inv(y,z,x')\omega(y,x,z)\alpha_x\inv(z,z\inv)\\
              &=\omega\inv(y,z,x')\omega(y,x,z)\alpha_{x'}\inv(z\inv,z)\\
            \Psi(u\ot v\ot w)&=(|u|\hit v\ot w)\ot |w|\inv\hit u
\end{align*}
Further
\begin{multline*}
  R(\braid_2\inv\braid_1)((x\hit v\ot w)\ot z\inv\hit u)\\
  =\psi(x\hit y,(z\inv\hit x)\inv\hit(x\hit y),(x\hit y)\hit z,z\inv\hit x)\Psi((x\hit v\ot w)\ot z\inv\hit u)
  \\ 
  =\psi(x\hit y,y,y\hit z,z\inv\hit x)\Psi((x\hit v\ot w)\ot z\inv\hit u)\\
\end{multline*}
where
\begin{multline*} \Psi((x\hit v\ot w)\ot z\inv\hit u)\\
  =\alpha((z\inv\hit x)\inv, x)((x\hit y)\hit w\ot z\inv\hit u)\ot [z\inv,x\inv]\hit v\\
  \in W_{(x\hit y)\hit z}\ot U_{z\inv\hit x}\ot V_{[z\inv,x\inv]\hit y}=W_{y\hit z}\ot U_{z\inv\hit x}\ot V_y.
\end{multline*}
Further,
\begin{multline*}
  R(\braid_2\inv\braid_1)(((x\hit y)\hit w\ot z\inv\hit u)\ot[z\inv,x\inv]\hit v)\\
  =\psi(y\hit z,y\inv\hit(y\hit z),(y\hit z)\hit(z\inv\hit x),y)\Psi((x\hit y)\hit w\ot z\inv\hit u\ot[z\inv,x\inv]\hit v)\\
  =\psi(y\hit z,z,x,y)\Psi((x\hit y)\hit w\ot z\inv\hit z\ot[z\inv,x\inv]\hit v)\\
\end{multline*}
\begin{multline*}
\Psi((x\hit y)\hit w\ot z\inv\hit z\ot[z\inv,x\inv]\hit v) \\
  =\alpha_x(y\hit z, z\inv)\alpha_z(y, x\hit y)\psi(y\hit z,z,x,y)([y,z]\hit u\ot [z\inv,x\inv]\hit v\ot[y,x]\hit w.
\end{multline*}
Thus
\begin{multline*}
  R((\braid_2\inv\braid_1)^3((u\ot v)\ot w)\\
  =\Omega(x,y,z)[y,z]\hit u\ot[z\inv,x\inv]\hit v\ot [y,x]\hit w
\end{multline*}
with
\begin{multline}\label{omega}
\Omega(x,y,z)=\psi(x,z\inv\hit x,x\hit y,z)\psi(x\hit y,y,y\hit z,z\inv\hit x)\psi(y\hit z,z,x,y)\\
  =\omega(x\hit y,z,z\inv\hit x)\omega\inv(x\hit y,x,z)\alpha_x\inv(z,z\inv)\\
  \omega(y\hit z,z\inv\hit x,y)\omega\inv(y\hit z,x\hit y,z\inv\hit x)\alpha_y\inv(z\inv\hit x\inv,z\inv\hit x)\\
  \omega(x,y,z)\omega\inv(x,y\hit z,y)\alpha_z(y\inv,y)
\end{multline}

We conclude that
\begin{align*}
  B&_{(g,\chi_1),(h,\chi_2),(k,\chi_3)}\\&=\tr(R((\sigma_2\inv\sigma_1)^3))\\
  &=\tr(R((\sigma_2\inv\sigma_1)^3)_0)\\
  &=\sum_{\substack{x\in\ol g,y\in\ol h,z\in\ol k\\~(\cref{eq:condition0})}}\Omega(x,y,z)\chi_1^x([y,z])\chi_2^y([z\inv,x\inv])\chi_3^z([y,x]).
\end{align*}

Using \cref{trick}, we can reduce the three summations in the preceding formula to two.
\begin{align*}
  B&_{(g,\chi_1),(h,\chi_2),(k,\chi_3)}\\
   &=|\ol k|\sum_{\substack{x\in\ol g,y\in\ol h \\~(\cref{eq:condition1})}}\Omega(x,y,k)\chi_1^x([y,k])\chi_2^y([k\inv,x\inv])\chi_3([y,x])\\
   &=|\ol k|\sum_{\substack{x\in\ol g,y\in\ol h \\~(\cref{eq:condition1})}}
  \begin{aligned}[t]
    \Omega&(x,y,k)\alpha_{g} ([y, k], p)\alpha^{-1}_{g}(p, p^{-1} \hit [y, k])\\
    &\alpha_{h}([k^{-1}, x^{-1}],q) \alpha^{-1}_{h} (q, q^{-1} \hit [k^{-1},x^{-1}])\\
    &\chi_1^x([y,k])\chi_2^y([k\inv,x\inv])\chi_3([y,x]).
  \end{aligned}
\end{align*}
with
\begin{equation}\label{eq:condition1}
  \begin{aligned}[c]
    [[k, y], x] &= 1\\
    [[y^{-1}, x],k]&=1
  \end{aligned}
\end{equation}
Finally, we can express the characters $\chi_1^x$, $\chi_2^y$ in terms of $\chi_1,\chi_2$ using (\cref{chi}):
\begin{align}
  B&_{(g,\chi_1),(h,\chi_2),(k,\chi_3)}\notag\\
   &=|\ol k|\sum_{\substack{p\in G/C_G(g)\\q\in G/C_G(g)\\~(\cref{eq:condition3})}}
  \begin{aligned}[t]\Omega(p\hit g&,q\hit h,k)(p\hit\chi_1)([q\hit h,k])\\
    &(q\hit \chi_2)([k\inv,p\hit g\inv])\chi_3([q\hit h,p\hit g])  
  \end{aligned}\\
   &=\frac{|\ol g||\ol h||\ol k|}{|G|^2}\sum_{\substack{p\in G\\q\in G
  \\~(\cref{eq:condition3})}}\label{eq:1}
  \begin{aligned}[t]\Omega(p\hit g&,q\hit h,k)(p\hit\chi_1)([q\hit h,k])\\
    &(q\hit \chi_2)([k\inv,p\hit g\inv])\chi_3([q\hit h,p\hit g])  
  \end{aligned}
\end{align}
with
\begin{equation}
\label{eq:condition3}
\begin{aligned}[c]
 [[k, q\hit h], p\hit g] &= 1\\ 
 [[(q\hit h)^{-1}, p\hit x],k]&=1.
\end{aligned}
\end{equation}
Besides the $\omega$ and $\alpha$ terms already hiding in $\Omega$, the conjugated characters in the last formulae are hiding further $\alpha$ terms from (\cref{chi}):
\begin{align*}
  (p\hit \chi_{1})([q\hit h, k]) &=\begin{aligned}[t]\alpha_{g} ([q\hit h, k], p) \alpha^{-1}_{g}(p, p^{-1} \hit [q\hit h, k])\\\cdot\chi_1(p^{-1} \hit [q\hit h, k])&
  \end{aligned}
  \\
(q\hit \chi_{2})([k^{-1}, p\hit g^{-1}])&=\begin{aligned}[t] \alpha_{h}([k^{-1}, p\hit g^{-1}],q) \alpha^{-1}_{h} (q, q^{-1} \hit [k^{-1}, p\hit g^{-1}])&\\\cdot\chi_2(q^{-1}\hit[k^{-1}, p\hit g^{-1}])&
\end{aligned}
\end{align*}

We have implemented the general formula for the $B$-tensor above in GAP; the codes are in \cref{codes}. For explicit calculations by humans, the abundance of cocycle terms (six $\omega$ terms and three $\alpha$ terms gathered in $\Omega$ plus four $\alpha$ terms hiding in conjugated projective characters) and the tedious commutation conditions on group elements certainly make the formulas somewhat unwieldy. We will now describe particular circumstances where these problems do not occur and the formula simplifies drastically. The special cases will be very useful for the key example we will treat in the following section.

\begin{Prop} Suppose there is an abelian normal subgroup $A\triangleleft G$ such that  $\omega$ is inflated from the quotient $G/A$ , and $g,h\in A$. Then
  \begin{equation}
    \label{eq:simplified}
    B_{(g,\chi_1),(h,\chi_2),(k,\chi_3)}=\vert \overline k \vert \cdot \sum_{\substack{x\in \overline g\\y\in\overline h}}\chi_1(p\inv\hit[y,k])\chi_2(q\inv\hit[k\inv,x\inv]).
  \end{equation}
  where $p\hit g=x,q\hit h=y$ in the sum. Note also that $\chi_1,\chi_2$ are ordinary characters in this case.
\end{Prop}
\begin{proof}
  The conditions are tailored to ensure that $\Omega(x,y,z)=1$ in \ref{omega}, since all values of $\omega$ where one argument is conjugate to $g$ or $h$ are trivial. The same holds for the $\alpha$ terms used in (\cref{chi}), since commutators with one element from $A$ also lie in $A$. Also, $[[k,q\hit h],p\hit g]=1$ for all $q,p$ since $q\hit h\in A$, hence $[k,q\hit h]\in A$, and $p\hit g\in A$. Finally $[q\hit h\inv,p\hit g]=1$ since $q\hit h,p\hit g\in A$ which is abelian, and in particular $[[q\hit h\inv,p\hit g],k]=1$.
\end{proof}
\begin{Cor}\label{Cor:verysimplified}
  Assume the hypotheses of the preceding corollary, and in addition that there is a subgroup $Q\subset C_G(k)$ such that $g^Q=\overline g$ and $h^Q=\overline h$. This applies for example when $G=A\semdir Q$ is a semidirect product of abelian groups. Then
  \begin{equation}
    \label{eq:verysimplified}
    B_{(g,\chi_1),(h,\chi_2),(k,\chi_3)}=\frac{\vert \ol k \vert |A|}{|C_Q(g)||C_Q(h)|} \cdot\sum_{q\in Q}\chi_2(q\hit[h,k])\chi_1(q\inv\hit[k\inv,g\inv]).
  \end{equation}
\end{Cor}
\begin{proof}
  We have bijections $Q/C_Q(g)\ni p\mapsto p\hit g\in \ol g$ and $Q/C_Q(h)\ni q\mapsto q\hit h\in \ol h$, and therefore, abbreviating $N=|C_Q(g)||C_Q(h)|$,
  \begin{align*}
    B&_{(g,\chi_1),(h,\chi_2),(k,\chi_3)}\\&=\frac{\vert\ol k  \vert}N \cdot \sum_{p,q\in Q}\chi_1(p\inv\hit[q\hit h,k])\chi_2(q\inv\hit[k\inv,p\hit g\inv])\\
                                        &=\frac{\vert\ol k  \vert}N \cdot\sum_{p,q\in Q}\chi_1(p\inv q\hit[h,q\inv\hit k])\chi_2(q\inv p\hit[p\inv \hit k\inv,g\inv])\\
                                        &=\frac{\vert\ol k  \vert}N\cdot\sum_{p,q\in Q}\chi_1(p\inv q\hit[h,k])\chi_2(q\inv p\hit[k\inv,g\inv])
  \end{align*}
  which gives the desired result after reparametrization.
\end{proof}

\section{Twisted doubles of nonabelian groups of order $pq$}\label{sec:Example}
Let $p,q$ be odd primes with $p|q-1$. There is a unique nonabelian group of order $pq$, and there are exactly $p$ inequivalent twisted doubles of that group; see \cite{2017arXiv170802796M} and below. However, these $p$ inequivalent modular categories only afford three different sets of modular data. In this section we will show that the $T$-matrix and the borromean tensor do distinguish the $p$ modular categories. We first tested this for the case $p=5$ and $q=11$ using our GAP-implementation of~(\cref{eq:1}). That is, we computed the values of $T$-matrix, $S$-matrix, and $B$-tensor in this case, and verified that no bijection between the simples of two distinct categories maps all three data to each other. Closer inspection of the experimental data also helped us pick out the particular simples to use in our borromean tensor calculations below. Thus, although no computer help is necessary in the end to prove our results, machine calculations were instrumental in  our finding them. Note that the $S$-matrix turns out not to be necessary in our example in the end; the $T$-matrix and $B$-tensor suffice.

The nonabelian group of order $pq$ where $p$ and $q$ are odd primes such that $p | q-1$ has the following presentation:
$$
\ZZ/q\ZZ \semdir \ZZ/p\ZZ \cong \langle a,b | a^q=b^p=1 ,\ bab^{-1}=a^n \rangle
$$
The integer $n \in \ZZ/q\ZZ$ must be chosen such that $n \not\equiv 1 \mod q$ and $n^p \equiv 1 \mod q$, but the group does not depend on that choice.
We note
\begin{align}\label{eq:pqformulas}
  b^k a^l b^{-k} &= a^{n^k l} &
  [a^l,b^k]&=a^{l(1-n^k)}&
  [b^k,a^l]&=a^{l(n^k-1)}
\end{align}

The canonical surjection $\ZZ/q\ZZ \semdir \ZZ/p\ZZ\to \langle b\rangle\cong\ZZ/p\ZZ$ induces an isomorphism between the cohomology groups $H^3(\langle b\rangle ,\CCu)\to H^3(\ZZ/q\ZZ \semdir \ZZ/p\ZZ,\CCu)$, where $H^3(\langle b\rangle ,\CCu)\cong \ZZ/p\ZZ$ is generated by the following cocycle $\omega$:
$$
\omega(b^i,b^j,b^k):=\exp\left( \frac{2i\pi}{p^2}([i]([j]+[k]-[j+k])\right)
$$
where $[i]\in\{0,\dots,p-1\}$ is such that $[i]\equiv i\mod p$.

The following is a complete list of representatives for the isomorphism classes of simples of $\CTR(\Vect_G^{\omega^u})$:
\begin{enumerate}
\item
$(1,\chi)$ where $\chi$ is an irreducible character of $G$,
\item
$(a^l,\chi_q^s)$ where $l\in (\ZZ/q\ZZ)^\times/\langle n\rangle$, $s\in\ZZ/q\ZZ$,
and $\chi_q$ is the generator of $\widehat{\langle a \rangle}\cong \ZZ/q\ZZ$ given by
$$
\chi_q(a)=\exp\left(\frac{2i\pi}{q}\right),
$$ 
\item
$(b^k,\widetilde{\chi_p^r})$ where $k\in(\ZZ/p\ZZ)^\times$, $r\in \ZZ/p\ZZ$, and $\widetilde{\chi_p^r}$ is the $\alpha_{b^k}^u$-projective character of $C_G(b^k)=\widehat{\langle b \rangle}$ associated to $\chi_p^r$, $\chi_p$ being the generator of $\widehat{\langle b \rangle}\cong \ZZ/p\ZZ$ given by
$$
\chi_p(b)=\exp\left(\frac{2i\pi}{p}\right).
$$ 
That is $\widetilde{\chi_p^r}=\chi_p^r \mu_{b^k}^u$ where $\alpha_{b^k}^u=d \mu_{b^k}^u$. In the sequel, we will write the simple as $(b^k,\chi_p^r)$ instead of $(b^k,\widetilde{\chi_p^r})$.
\end{enumerate}
For simplicity, we will refer to these as simples of type one, type two or type three.
\begin{Lem}\label{TLem}
The $T$-matrix of $\CTR(\Vect_G^{\omega^u})$ is given by the following:
\begin{enumerate}
\item
$\Theta(1,\chi)=1$
\item
$\Theta(a^l,\chi_q^s)=\exp\left(\frac{2i\pi}{q}sl\right)$
\item
$\Theta(b^k,\chi_p^r)=\exp\left(\frac{2i\pi}{p^2}(pkr+k^2u)\right)$, and in particular
\item $\Theta(b^k,\chi_p^r)^p=\zeta_p^{k^2u}$ for $\zeta_p=\exp(\frac{2i\pi}{p})$.
\end{enumerate}
\end{Lem}
\begin{proof}
For simples in $\CTR(\Vect_G^{\omega^u})$ of the first type, it is obvious that the corresponding twist $\Theta(1,\chi)=\frac{\chi(1)}{\chi(1)}=1$. For simples of the second type $(a^l,\chi_q^s)$, one has $\alpha_{a^l}^u=1$, and then the irreducible projective characters of the centralizer $\langle a \rangle \cong \ZZ/q\ZZ$ of $a^l$ are the usual irreducible characters. The twist is therefore given by 
$$
\Theta(a^l,\chi_q^r)=\frac{\chi_q^s(a^l)}{\chi_q^s(1)}=\exp\left(\frac{2i\pi}{q}sl\right)
$$
Finally, for simples of the last type $(b^k,\chi_p^r)$, one has $\alpha_{b^k}^u=d \mu_{b^k}^u$ where
$$
\mu_{b^k}^u(x)=\exp\left(\frac{2i\pi u}{p^2}k[\Pi(x)]\right)
$$
and $\alpha_{b^k}^u$-projective characters of the centralizer $\langle b \rangle \cong \ZZ/q\ZZ$ of $b^k$ are usual irreducible characters twisted by $\mu_{b^k}^u$. Therefore:
$$
\Theta(b^k,\chi_p^r)=\frac{\chi_q^r(b^k)\mu_{b^k}^u(b^k)}{\chi_q^r(1)\mu_{b^k}^u(1)}=\exp\left(\frac{2i\pi}{p^2}(pkr+k^2u)\right)
$$
\end{proof}

\begin{Lem}\label{BProp} The $B$-tensor of the category $\CTR(\Vect_{\ZZ/q\ZZ \semdir \ZZ/p\ZZ}^{\omega^u})$ with $p$ and $q$ odd primes such that $q|p-1$ and $u\in\{0,\dots p-1\}$ satisfies
\begin{equation}\label{bformula.aab}
B_{(a^l,\chi_q^s),(a^l, \chi_q^s),(b^{k}, \chi_p^r)}=p q \sum_{m=0}^{p-1} e^{\left(\frac{2i\pi sl}{q} (n^{-t}-n^t)(n^m-n^{-m})\right)};\quad 2t\equiv k(p) 
\end{equation}
\end{Lem}

\begin{proof}
We apply \cref{Cor:verysimplified} with $Q=\langle b\rangle$ and $A=\langle a\rangle$.
By \cref{eq:verysimplified} and \cref{eq:pqformulas} we have
\begin{align*}
  B_{(a^l,\chi_q^s),(a^l, \chi_q^s),(b^k, \chi_p^r)} &=pq\sum_{m=0}^{p-1}\chi_q^s(b^m \hit[a^l,b^k])\chi_q^s(b^{-m} \hit[b^{-k},a^{-l}] )\\
  &=p q \sum_{m=0}^{p-1} \chi_q^s(a^{l \cdot n^m (1-n^k)})\chi_q^s(a^{-l \cdot n^{-m} (n^{-k}-1)})\\
&= p q \sum_{m=0}^{p-1} e^{\left(\frac{2i\pi sl}{q} (n^{m}(1-n^k)+n^{-m}(1-n^{-k})\right)}.
\end{align*}
For $2t\equiv k(p)$ we get
\begin{align*}
  n^{m}(1-n^k)+n^{-m}(1-n^{-k})&\equiv n^m(1-n^{2t})+n^{-m}(1-n^{-2t})\\&=(n^{m+t}-n^{-(m+t)})(n^{-t}-n^t)
\end{align*}
so reparametrization gives the desired expression.
\end{proof}

In \cite{2017arXiv170802796M} it was shown that the $p$ non-equivalent modular tensor categories  $\CTR(\Vect_{\ZZ/q\ZZ \semdir \ZZ/p\ZZ}^{\omega^u})$ for $u=0,..,p-1$, are not distinguished by their modular data. In fact there are only three different modular data between these categories. As we will see shortly, this is changed when we add the $B$-tensor to the data. 

The fact that the $B$-tensor should help may seem surprising given that the expression \eqref{bformula.aab} does not depend on $u$ at all. At least in the case $p=5, q=11$ this is not only true for the particular coefficients of the $B$-tensor that we calculated above. If we label simple objects by $\chi$, resp.\ $(l,s)$, resp.\ $(k,r)$ as above, then the entire $B$-tensor is independent of $u$, as we found by computer calculation. However, in this labelling, the $T$-matrices do depend on $u$, and for different $u$ we cannot relabel so that both $B$ and $T$ agree. We will inspect this closer for some low order examples.

In Tables 1, 2 and 3, we have listed, in the cases $(p,q,n)=(5,11,4)$, $(5,31,2)$, $(7,29, 7)$, the rescaled $B$-tensor values $$\beta_{ls}(k):=\frac{1}{pq}B_{(a^l,\chi_q^s),(a^l, \chi_q^s),(b^k, \chi_p^r)}$$ for a fixed value of the product $ls$ (they do not depend on $r$, and only half are listed since $\beta_{ls}(p-k)=\beta_{ls}(k)$). We have set $\zeta=\exp\left(\frac{2\pi isl}{q}\right)$, which is a primitive $q$-th root of unity if $ls\not\equiv 0\mod q$.
\begin{table}
\caption{$p=5, q=11, n=4$}
\begin{tabular}{|l|l|l|}
\hline
$k$ & $t$ & $\beta_{ls}(k)$ \\ \hline 
$1$ & $3$ &  $1+ \zeta^4 + \zeta^6 +  \zeta^5 +  \zeta^7$  \\ \hline
$2$ & $1$ & $1 + \zeta + \zeta^7 + \zeta^4 + \zeta^{10}$ \\ \hline
\end{tabular}
\end{table}

\begin{table}
\caption{$p=5, q=31, n=2$}
\begin{tabular}{|l|l|l|}
\hline
$k$ & $t$ & $\beta_{ls}(k)$ \\ \hline 
$1$ & $3$ &  $1+ \zeta^6 + \zeta^{15} +  \zeta^{16} +  \zeta^{25}$  \\ \hline
$2$ & $1$ & $1 + \zeta^{10} + \zeta^{25} + \zeta^6 + \zeta^{21}$ \\ \hline
\end{tabular}
\end{table}

\begin{table}
\caption{$p=7, q=29, n=7$}
\begin{tabular}{|l|l|l|}
\hline
$k$ & $t$ &$ \beta_{ls}(k)$ \\ \hline
$1$ & $3$ &  $1 + \zeta^4 + \zeta^{11} +  \zeta +  \zeta^{28} +\zeta^{25} +\zeta^{18}$  \\ \hline
$2$ & $1$ & $1 + \zeta^{24} + \zeta^{14} + \zeta^{18} + \zeta^{11} +\zeta^{15} +\zeta^5$ \\ \hline
$3$ & $4$ & $1 + \zeta^{15} + \zeta^4 +\zeta^{16} + \zeta^{25} + \zeta^{14} + \zeta^{13}$\\  \hline
\end{tabular}
\end{table}

We see that in the examples (and, it will turn out below, in general) the value of $k^2$ is determined by the entry of the $B$-tensor. A bijection between simples of categories for different values of $u$ that preserves the $T$-matrix would have to fix $ls$, and it would need to nontrivially permute the values of $k$ to compensate the difference in $u$, a contradiction. Details are in the proof of the following
\begin{Thm}\label{mainthm}
The $T$-matrix and $B$-tensor form a complete set of invariants for the $p$ non-equivalent modular categories $\CTR(\Vect_{\ZZ/q\ZZ \semdir \ZZ/p\ZZ}^{\omega^u})$ where $p$ and $q$ are odd primes such that $p|q-1$. More precisely, if there is a map $\kappa$ from the simple objects of $\CTR(\Vect_{\ZZ/q\ZZ \semdir \ZZ/p\ZZ}^{\omega^u})$ to the simple objects of $\CTR(\Vect_{\ZZ/q\ZZ \semdir \ZZ/p\ZZ}^{\omega^{u'}})$ satisfying $T_{\kappa(V)}=T_{V}$ and $B_{\kappa(U),\kappa(V),\kappa(W)}=B_{UVW}$ for all simples $U,V,W$ of the former, then $u=u'$.
\end{Thm}
\begin{proof}
  Let $\kappa$ be such a map. We will use the notations $(g,\chi)_u$ and $(g,\chi)_{u'}$ to denote simple objects of the two categories under consideration. The category corresponding to the trivial cocycle is easily distinguished from the others by the $T$-matrix alone, so we can assume $u,u'\neq 0$. Let $ls\not\equiv 0\mod q$. Then it is obvious from \cref{TLem} that $\kappa((a^l,\chi_q^s)_u)=(a^{l'},\chi_q^{s'})_{u'}$ with $sl \equiv s'l' \mod q$. Also, part (4) of \ref{TLem} implies that $\kappa((b^k,\chi_p^r)_u)=(b^{k'},\chi_p^{r'})_{u'}$ for some $k'$ and $r'$ with $k'^2u'\equiv k^2u\mod p$.

Next,
\begin{align*}
B_{(a^l,\chi_q^s)_u,(a^l, \chi_q^s)_u,(b^k, \chi_p^r)_u} &= B_{\kappa((a^l,\chi_q^s)_u),\kappa((a^l, \chi_q^s)_u),\kappa((b^k, \chi_p^r))}\\
&= B_{(a^{l'},\chi_q^{s'})_{u'},(a^{l'},\chi_q^{s'})_{u'},(b^{k'}, \chi_p^{r'})_{u'}}
\end{align*}
So, with \cref{BProp} and $2t\equiv k\mod p, 2t'\equiv k'\mod p$, we get
\begin{equation}\label{sumroots}
\sum_{m=0}^{p-1} \left(e^{\frac{2i\pi sl} {q} }\right)^{(n^t-n^{-t})(n^{m}-n^{-m})} = \sum_{m=0}^{p-1} \left(e^{\frac{2i\pi s'l'} {q} }\right)^{(n^{t'}-n^{-t'})(n^{m}-n^{-m})}.
\end{equation}
Since this is a $\mathbb Q$-linear relation between fewer than $q$ powers of the same primitive $q$-th root of unity $e^{\frac{2i\pi sl} {q} }=e^{\frac{2i\pi s'l'} {q}}$, we conclude that the set 
\begin{equation*}
  M_t:=\{(n^t-n^{-t})(n^{m}-n^{-m})|m=0,\dots,p-1\}\subset\ZZ/q\ZZ
\end{equation*}
is equal to the analogous set $M_{t'}$. We note that the $p$ elements $n^m-n^{-m}$ are distinct: Indeed, assume $n^m-n^{-m}=n^j-n^{-j}$ for $0\leq m,j\leq p-1$. Then
  \begin{equation*}
    0=n^m-n^j+n^{-j}-n^{-m}=(n^{-j-m}+1)(n^m-n^j).
  \end{equation*}
 Now $n^{-j-m}+1\neq 0$ since $n$ has odd order, and so $n^m=n^j$ which implies $m=j$. In particular
\begin{align*}
  \sum_{x\in M_t}x^2&=\sum_{m=0}^{p-1}(n^{t}-n^{-t})^2(n^{2m}+n^{-2m}-2)\\
                  &=(n^{t}-n^{-t})^2\left(\sum_{m=0}^{p-1}n^{2m}+\sum_{m=0}^{p-1}n^{-2m}-2p\right)\\
  &=-2p(n^{t}-n^{-t})^2,
\end{align*}
since $n^2$ and $n^{-2}$ are primitive $p$-th roots of unity in $\ZZ/q\ZZ$
and so the sum over all their powers gives zero. By the same reasoning
\begin{equation*}
  \sum_{x\in M_t}x^2=\sum_{x\in M_{t'}}x^2=-2p(n^{t'}-n^{-t'})^2
\end{equation*}
Thus $(n^t-n^{-t})^2=(n^{t'}-n^{-t'})^2$, which implies $n^t-n^{-t}=n^{t'}-n^{-t'}$ or $n^t-n^{-t}=n^{-t'}-n^{t'}$. As we have seen, this implies $t\equiv t'$ or $t\equiv -t'$ in $\ZZ/p\ZZ$. Therefore $k^2=k'^2$ and we can conclude that $u=u'$.
\end{proof}

\begin{Rem}
  There is an action of the absolute Galois group of abelian extensions  of the rationals $\Gamma:=\Gal(\Q^\ab/\Q)$ on the set of simples of any integral modular category; see \cite[Appendix]{MR2183279}. In the case of $\CTR(\Vect_G^\omega)$ it satisfies $S_{\gamma(i),\gamma(j)}=\gamma^2(B_{i,j})$ for any $\gamma\in\Gamma$, and also $T_{\gamma(i),\gamma(i)}=\gamma^2(T_{ii})$ as shown in  \cite{MR3435813}. If we analyze the proof in \cite{2017arXiv170802796M} that the twisted doubles of nonabelian groups of order $pq$ share only three different sets of modular data, we can conclude from our result that the analogous property $B_{\gamma(i),\gamma(j),\gamma(k)}=\gamma^2(B_{ijk})$ is not satisfied. 
\end{Rem}

\appendix

\section{GAP codes}\label{codes}
We give in this appendix the codes that we used to compute the $S$-matrix and the Borromean tensor for the categories $\CTR(\Vect_G^\omega)$. Preliminary codes that compute complex valued group cohomology, projective characters of finite groups, etc., as well as the code that computes the $T$-matrix, are the ones of \cite{2017arXiv170806538M}. The function \lstinline!ZwG_S! computes the $S$-matrix and the function \lstinline!ZwG_B! computes the Borromean tensor; both those functions are taking as arguments a finite group $G$, a $3$-cocycle $\omega \in \in H^3(G,\CCu)$, the simple objects of $\CTR(\Vect_G^\omega)$ and a non-negative integer $e$. More precisely, $\omega$ is given as its list of values in $\ZZ/e\ZZ$, where $e$ is the exponent of $H^3(G,\CCu)$ and the simple objects are couples $(g,\chi)$ where $g \in G$ and $\chi$ is a projective character given by its list of values (in $\CC$) on the centralizer $C_G(g)$.\\

\begin{lstlisting}
ZwG_S := function(G, w, Simples, e)
local ord, listG, alphag, aval, bval, lista, 
      listb, s, a, b, g, simple1, simple2, x;
ord:= Size(G);
listG:=EnumeratorSorted(G);
alphag:=function(g)
 return function(x,y)
  return Alpha_symb(G,w,listG[g]) 
          (listG[x],listG[y]);
 end;
end;
aval:=[];
for simple1 in Simples do
 a:=simple1!.class;
 lista:=EnumeratorSorted(Centralizer(G,listG[a]));
 bval:=[];
 for simple2 in Simples  do
  b:=simple2!.class;
  listb:=EnumeratorSorted(Centralizer(G,listG[b]));
  s := 0;
  for g in [1..ord] do
   if not 
    Commm(Conjugation(listG[g],listG[a]),listG[b])
    =One(G) 
   then continue;
   fi;
   s := s + 
    E(e)^(
     alphag(a)( b, g) 
     - alphag(a)(g,Position(
        listG,Conjugation(listG[g]^-1, listG[b])))
    ) 
    * simple1!.chi[Position(
       lista, Conjugation(listG[g]^-1, listG[b]))]
    * simple2!.chi[Position(
       listb, Conjugation(listG[g], listG[a]))];
  od;
  Add(bval, 
   s 
   * Size(ConjugacyClass(G,listG[b]))
   / Size(Centralizer(G, listG[a]))
   );  
 od;
 Add(aval, bval);
od;
return aval;
end;

\end{lstlisting}
\begin{lstlisting}



ZwG_B:=function(G,cocyclevalues,Simples,e)
 local 	conjG,listG,posG,alphag,simple1,simple2,
        simple3,a,b,c,chia,chib,chic,lista,listb,
        listc,tensor,matrixb,rowc,sum,g,h,ap,bp,
        ap_inv,bp_inv,c_inv,cinv_hit_ap,ap_hit_bp,
        bp_hit_c,cinv_hit_apinv,commm_cinv_apinv,
        hinv_hit_commca,commm_bp_c,
        ginv_hit_commmbc,commm_bpinv_ap;
 listG:=EnumeratorSorted(G);
 posG:=function(g)
  return Position(listG,g);
 end;
 alphag:=function(g)
  return function(x,y)
   return Alpha_symb(G,cocyclevalues,listG[g]) 
           (listG[x],listG[y]);
   end;
 end;
 tensor:=[];
 for simple1 in Simples do
  a:=simple1!.class;
  lista:=EnumeratorSorted(Centralizer(G,listG[a]));
  matrixb:=[];
  for simple2 in Simples do
   b:=simple2!.class;
   listb:=EnumeratorSorted(Centralizer(G,listG[b]));
   rowc:=[];
   for simple3 in Simples do
    c:=simple3!.class;
    listc:=EnumeratorSorted(Centralizer(G,listG[c]));
    sum:=0;
    for g in [1..Size(listG)] do
     for h in [1..Size(listG)] do
      ap:=posG(Conjugation(listG[g],listG[a]));
      bp:=posG(Conjugation(listG[h],listG[b]));
      ap_inv:=posG(listG[ap]^-1);
      bp_inv:=posG(listG[bp]^-1);
      c_inv:=posG(listG[c]^-1);
      cinv_hit_ap:=
       posG(Conjugation(listG[c_inv],listG[ap]));
      ap_hit_bp:=
       posG(Conjugation(listG[ap],listG[bp]));
      bp_hit_c:=
       posG(Conjugation(listG[bp],listG[c]));
      cinv_hit_apinv:=
       posG(Conjugation(
        listG[c_inv],listG[ap_inv]));
      commm_cinv_apinv:=
       posG(Commm(listG[c_inv],listG[ap_inv]));
      hinv_hit_commca:=
       posG(Conjugation(
        listG[h]^-1,listG[commm_cinv_apinv]));
      commm_bp_c:=
       posG(Commm(listG[bp],listG[c]));
      ginv_hit_commmbc:=
       posG(Conjugation(
        listG[g]^-1,listG[commm_bp_c]));
      commm_bpinv_ap:=
       posG(Commm(listG[bp_inv],listG[ap]));
      if not 
       (Commm(Commm(listG[bp_inv],listG[ap]),
                    listG[c])
       =One(G) 
       and 
       Commm(Commm(listG[bp],listG[c]),
                   listG[ap])
       =One(G) ) 
      then continue;
      else
       sum:=sum + E(e) ^ (
		cocyclevalues[ap_hit_bp][ap][c]  
	-cocyclevalues[ap_hit_bp][c][cinv_hit_ap] 
	+cocyclevalues[bp_hit_c][ap_hit_bp]
	                          [cinv_hit_ap] 
	-cocyclevalues[bp_hit_c][cinv_hit_ap][bp] 
	+cocyclevalues[ap][bp_hit_c][bp] 
	-cocyclevalues[ap][bp][c] 
	-alphag( ap )  ( c , c_inv ) 
	+alphag( ap )  ( bp_hit_c , c_inv )  
	-alphag( ap_hit_bp )  
	    ( cinv_hit_ap , cinv_hit_apinv ) 
	+alphag( bp )  ( cinv_hit_apinv , ap ) 
	-alphag( bp_hit_c )  ( bp , bp_inv )  
	+alphag( c )  ( bp_inv , ap_hit_bp )  
	+alphag( b )  ( commm_cinv_apinv , h )  
	-alphag( b )  ( h , hinv_hit_commca )  
	+alphag( a )  ( commm_bp_c , g )  
	-alphag( a )  ( g , ginv_hit_commmbc ) ) 
	*simple1!.chi[ 
	  Position(lista,listG[ginv_hit_commmbc] )]   
	*simple2!.chi[ 
	  Position(listb,listG[hinv_hit_commca] )]  
	*simple3!.chi[ 
	  Position(listc,listG[commm_bpinv_ap] )] ;
      fi;
     od;
    od;
    sum:=sum * 
    ( Size(ConjugacyClass(G,listG[c])) 
    / ( Size(Centralizer(G,listG[a])) 
    * Size(Centralizer(G,listG[b])) ) );
    Add(rowc,sum);
   od;
   Add(matrixb,rowc);
  od;
  Add(tensor,matrixb);
 od;
 return tensor;
end;
\end{lstlisting}
Finally, checking whether one or more matrices or ``tensors'' indexed by a power of the same index set are identical up to a permutation of the index set (i.~e.\ simultaneous permutations of the matrix or tensor indices) is in itself a tricky task. We include a function that does this for  $T$, $S$ and $B$ (based on an analogous function that we wrote for the modular data) using some heuristic tricks to speed up the procedure.
\begin{lstlisting}
Same_S_T_B:=function(S1,T1,B1,S2,T2,B2)
 local l,P,lastbad,n,i,j,A,Q,blocks,PS1,PS2,rev;
 lastbad:=function(S1,T1,B1,S2,T2,B2,P,Q,l)
  local j,k;
   if P[l] in List([1..l-1],i->P[i]) then return true;
   fi;
   if T2[P[l]] <>  T1[Q[l]] then return true ;
   fi;
   for k in [1..l] do
    if S1[Q[k]][Q[l]]<>S2[P[k]][P[l]] then return true;
    fi;
   od;
   for j in [1..l] do
    for k in [1..l] do
     if B1[Q[j]][Q[k]][Q[l]]<>B2[P[j]][P[k]][P[l]]
      then return true;
     fi;
    od;
   od;
   return false;
 end;
 presorted:=function(S,T,B)
  local labels,labelset,perm,n,blocks,TS,SS,BS;
  n:=Size(T);
  labels:=List([1..n],
               i->[T[i],S[i][i],Collected(S[i])]);
  labelset:=Set(labels);
  perm:=[1..n];
  SortParallel(labels,perm);
  TS:=List(perm,i->T[i]);
  SS:=List(perm,i->List(perm,j->S[i][j]));
  BS:=List(perm,
           i->List(perm,
                   j->List(perm,
                           k->B[i][j][k])));
  blocks:=List(labelset,
               l->Filtered([1..n],j->labels[j]=l));
  return [SS,TS,BS,perm,blocks,labels];
 end;
 n:=Size(T1);
 PS1:=presorted(S1,T1,B1);
 PS2:=presorted(S2,T2,B2);
 if PS1[2]<>PS2[2] then
  return [false,"not the same T"];
 fi;
 if List(PS1[6],x->x[2])<>List(PS2[6],x->x[2]) then
    return [false,"T and diag(S) don't sort parallelly"];
 fi;
 if PS1[5]<>PS2[5]
  then return [false,"not the same blocks"];
 fi;
 if PS1[6]<>PS2[6]
  then return [false,"unsorted data don't match"];
 fi;
   
 blocks:=PS1[5];
 rev:=[];
 S1:=PS1[1];
 T1:=PS1[2];
 B1:=PS1[3];
 S2:=PS2[1];
 T2:=PS2[2];
 B2:=PS2[3];
 for i in [1..Size(blocks)] do
  for j in [1..Size(blocks[i])] do
      rev[blocks[i][j]]:=[i,j];
  od;
od;
nextinblock:=function(i)
 if rev[i][2]=Size(blocks[rev[i][1]]) then return n+1;
 fi;
 return i+1;
end;
Q:=[];
for i in [1..Maximum(List(blocks,Size))] do
 A:=List(Filtered(blocks,b->Size(b)>=i),b->b[i]);
 Q:=Concatenation(Q,A);
od;
l:=1;
P:=[Q[1]];
while true do
 if P[l]>n then
  l:=l-1;
  if l=0 then return false;
  fi;
  Remove(P);
  P[l]:=nextinblock(P[l]);
  continue;
 fi;
 if lastbad(S1,T1,B1,S2,T2,B2,P,Q,l) then
  P[l]:=nextinblock(P[l]);
  continue;
 fi;
 if l=n then return true;
 fi;
 l:=l+1;
 P[l]:=blocks[rev[Q[l]][1]][1];
od;       
end;
\end{lstlisting}
\bibliographystyle{alpha}
\bibliography{eigene,andere,arxiv,mathscinet}

\newcommand{\etalchar}[1]{$^{#1}$}
\def\cprime{$'$} \def\cprime{$'$} \def\germ{\mathfrak}\def\cprime{$'$}
  \def\cfgrv#1{\ifmmode\setbox7\hbox{$\accent"5E#1$}\else
  \setbox7\hbox{\accent"5E#1}\penalty 10000\relax\fi\raise 1\ht7
  \hbox{\lower1.05ex\hbox to 1\wd7{\hss\accent"12\hss}}\penalty 10000
  \hskip-1\wd7\penalty 10000\box7} \def\cprime{$'$}
  \def\cfgrv#1{\ifmmode\setbox7\hbox{$\accent"5E#1$}\else
  \setbox7\hbox{\accent"5E#1}\penalty 10000\relax\fi\raise 1\ht7
  \hbox{\lower1.05ex\hbox to 1\wd7{\hss\accent"12\hss}}\penalty 10000
  \hskip-1\wd7\penalty 10000\box7} \def\cprime{$'$}
  \def\cfgrv#1{\ifmmode\setbox7\hbox{$\accent"5E#1$}\else
  \setbox7\hbox{\accent"5E#1}\penalty 10000\relax\fi\raise 1\ht7
  \hbox{\lower1.05ex\hbox to 1\wd7{\hss\accent"12\hss}}\penalty 10000
  \hskip-1\wd7\penalty 10000\box7}
\begin{thebibliography}{BNRW16}

\bibitem[BDG{\etalchar{+}}18]{2018arXiv180505736B}
P.~{Bonderson}, C.~{Delaney}, C.~{Galindo}, E.~C. {Rowell}, A.~{Tran}, and
  Z.~{Wang}.
\newblock {On invariants of Modular categories beyond modular data}.
\newblock {\em ArXiv e-prints}, May 2018.

\bibitem[BDSV15]{2015arXiv150906811B}
B.~{Bartlett}, C.~L. {Douglas}, C.~J. {Schommer-Pries}, and J.~{Vicary}.
\newblock {Modular categories as representations of the 3-dimensional bordism
  2-category}.
\newblock {\em ArXiv e-prints}, September 2015.

\bibitem[BK01]{BakKir:LTCMF}
Bojko Bakalov and Jr.~Alexander Kirillov.
\newblock {\em Lectures on tensor categories and modular functors}, volume~21
  of {\em University Lecture Series}.
\newblock American Mathematical Society, Providence, RI, 2001.

\bibitem[BNRW16]{MR3486174}
Paul Bruillard, Siu-Hung Ng, Eric~C. Rowell, and Zhenghan Wang.
\newblock Rank-finiteness for modular categories.
\newblock {\em J. Amer. Math. Soc.}, 29(3):857--881, 2016.

\bibitem[CGR00]{MR1770077}
Antoine Coste, Terry Gannon, and Philippe Ruelle.
\newblock Finite group modular data.
\newblock {\em Nuclear Phys. B}, 581(3):679--717, 2000.

\bibitem[DLN15]{MR3435813}
Chongying Dong, Xingjun Lin, and Siu-Hung Ng.
\newblock Congruence property in conformal field theory.
\newblock {\em Algebra Number Theory}, 9(9):2121--2166, 2015.

\bibitem[DPR90]{MR1128130}
R.~Dijkgraaf, V.~Pasquier, and P.~Roche.
\newblock Quasi {H}opf algebras, group cohomology and orbifold models.
\newblock {\em Nuclear Phys. B Proc. Suppl.}, 18B:60--72 (1991), 1990.
\newblock Recent advances in field theory (Annecy-le-Vieux, 1990).

\bibitem[EGNO15]{MR3242743}
Pavel Etingof, Shlomo Gelaki, Dmitri Nikshych, and Victor Ostrik.
\newblock {\em Tensor categories}, volume 205 of {\em Mathematical Surveys and
  Monographs}.
\newblock American Mathematical Society, Providence, RI, 2015.

\bibitem[ENO05a]{EtiNikOst:FC}
Pavel Etingof, Dmitri Nikshych, and Viktor Ostrik.
\newblock On fusion categories.
\newblock {\em Ann. of Math. (2)}, 162(2):581--642, 2005.

\bibitem[ENO05b]{MR2183279}
Pavel Etingof, Dmitri Nikshych, and Viktor Ostrik.
\newblock On fusion categories.
\newblock {\em Ann. of Math. (2)}, 162(2):581--642, 2005.

\bibitem[Maj98]{Maj:QDQHA}
S.~Majid.
\newblock Quantum double for quasi-{H}opf algebras.
\newblock {\em Lett. Math. Phys.}, 45:1--9, 1998.

\bibitem[MS17a]{2017arXiv170802796M}
M.~{Mignard} and P.~{Schauenburg}.
\newblock {Modular categories are not determined by their modular data}.
\newblock {\em ArXiv e-prints}, August 2017.

\bibitem[MS17b]{2017arXiv170806538M}
M.~{Mignard} and P.~{Schauenburg}.
\newblock {Morita equivalence of pointed fusion categories of small rank}.
\newblock {\em ArXiv e-prints}, August 2017.

\bibitem[NR11]{MR2832261}
Deepak Naidu and Eric~C. Rowell.
\newblock A finiteness property for braided fusion categories.
\newblock {\em Algebr. Represent. Theory}, 14(5):837--855, 2011.

\end{thebibliography}
\end{document}